\documentclass[leqno]{article}
  \usepackage{amsmath,amsthm,amssymb,tikz-cd,mathtools,caption,cleveref,pgfplots,tikz,here,enumitem}
  \usetikzlibrary{positioning}
  \usetikzlibrary{arrows}
  \theoremstyle{definition}
  \newtheorem{thm}{Theorem}[section]
  
  \crefname{thm}{Theorem}{Theorems}
  \newtheorem{prop}[thm]{Proposition}
  
  \crefname{prop}{Proposition}{Propositions}
  \newtheorem{lem}[thm]{Lemma}
  
  \crefname{lem}{Lemma}{Lemmas}

  \newtheorem{exam}[thm]{Example}
  
  \crefformat{enumi}{(#2#1#3)}
  \newcommand{\Aut}{\mathop{\mathrm{Aut}}\nolimits}
  \newcommand{\diag}{\mathop{\mathrm{diag}}\nolimits}
  \newcommand{\ad}{\mathop{\mathrm{ad}}\nolimits}
  \newcommand{\id}{\mathrm{id}}
  \newcommand{\Deltare}{\Delta^{\mathrm{re}}}
  
  \newcommand{\lie}[1]{\mathfrak{#1}}
  \newcommand{\bb}[1]{\mathbb{#1}}
  \newcommand{\bbZ}{\bb{Z}}
  \newcommand{\bbR}{\bb{R}}
  \newcommand{\bbC}{\bb{C}}
  \newcommand{\mcl}[1]{\mathcal{#1}}
  \newcommand{\disp}{\displaystyle}
  \newcommand{\vep}{\varepsilon}
  \newcommand{\lig}{\lie{g}}
  \newcommand{\lih}{\lie{h}}
  \newcommand{\lisl}{\lie{sl}}
  \newcommand{\lin}{\lie{n}}
  \newcommand{\hsp}{\hspace{2em}}
  \newcommand{\lra}[1]{\langle #1 \rangle}
  \newcommand{\norm}[1]{\lvert #1 \rvert}
  \newcommand{\ovl}[1]{\overline{#1}}
  \allowdisplaybreaks[4]

  \newlist{newEnum}{enumerate}{3}
  \setlist[newEnum,1]{label=(\roman*),
                    ref  =(\roman*)}
  \setlist[newEnum,2]{label=(\alph*),
                    ref  =\themyenumi{.(\alph*)}}
  \setlist[newEnum,3]{label=(\arabic*),
                    ref  =\themyenumii{.(\arabic*)}}
  \crefname{newEnumi}{}{}
  \crefname{newEnumii}{}{}
  \crefname{newEnumiii}{}{}

\title{$\lisl_2$ triples whose nilpositive elements
are in a space which is spanned by the real root vectors
in rank 2 symmetric hyperbolic Kac-Moody Lie algebras}
\author{TSURUSAKI Hisanori\thanks{Graduate School of Mathematical Sciences, University of Tokyo,
htsuru@ms.u-tokyo.ac.jp}}
\date{}
    
\begin{document}
\maketitle
\begin{abstract}
  In analogy to the theory of nilpotent orbit in finite-dimensional semisimple Lie algebras,
  it is known that the principal $\mathfrak{sl}_2$ subalgebras
  can be constructed in hyperbolic Kac-Moody Lie algebras. 
  We obtained a series of $\mathfrak{sl}_2$ subalgebras
  in rank 2 symmetric hyperbolic Kac-Moody Lie algebras
  by extending the aforementioned construction.
  We present this result and also discuss $\mathfrak{sl}_2$ modules
  obtained by the action of the $\mathfrak{sl}_2$ subalgebras
  on the original Lie algebras.
\end{abstract}
\section{Introduction}
Kac-Moody Lie algebras are generalizations of finite-dimensional simple Lie algebras,
and can be divided into three types: finite type, affine type, and indefinite type
(\cite{kac}). Finite type Kac-Moody Lie algebras are finite-dimensional
simple Lie algebras.

In \cite{dyn}, nilpotent orbits in finite type Kac-Moody Lie algebras
were classified. Nilpotent orbits in a finite type Kac-Moody Lie algebra $\lig$
are those formed by inner automorphisms acting on nilpotent elements.
These orbits are completely classified by weighted Dynkin diagrams.

Moreover, from the Jacobson-Morosov theorem,
for a nilpotent element $x$, we can construct
an $\lisl_2$-triple whose nilpositive element is $x$ (\cite[Theorem 3.3.1]{cm}).
Classifying nilpotent orbits is equivalent to classifying $\lisl_2$-triples
in $\lig$ up to inner automorphisms.

Among the nilpotent orbits in a finite type Lie algebra,
a nilpotent orbit with the largest dimension (as an algebraic variety) is called
a principal nilpotent orbit.
Let $\{ e_i, f_i, h_i \}$ be the Chevalley generators of a finite type Lie algebra.
For each $i$, take the appropriate $a_i$'s in $\bbC$ and put
\begin{align*}
  X &= \sum_i e_i,\\
  Y &= \sum_i a_i f_i,\\
  H &= \sum_i a_i h_i.
\end{align*}
This $\{ X, Y, H \}$ is called an $\lisl_2$-triple corresponding to the
principal nilpotent orbit (\cite[Theorem 4.1.6]{cm}).

For a principal nilpotent orbit of a finite type Lie algebra,
we can create a principal $SO(3)$ subalgebra compatible with
the compact involution (\cite{kos}).
We are going into more details. 
Denote the Cartan matrix of $\lig$ by $A = (A_{ij})$
and the Chevalley generators by $\{ e_i, f_i, h_i \}$. Put
\begin{align*}
  p_i &= \sum_j A_{ij}^{-1} (> 0),\\
  n_i &= \sqrt{p_i}.
\end{align*}
Finally, put
\begin{align*}
  J_3 &= \sum_j p_j h_j,\\
  J^+ &= \sum_j n_j e_j,\\
  J^- &= \sum_j n_j f_j.
\end{align*}
This $\{ J_3, J^+, J^- \}$ spans a principal $SO(3)$ subalgebra.
This means that
\begin{align*}
  [J_3, J^\pm] &= \pm J^\pm,\\
  [J^+, J^-] &= J_3.
\end{align*}
In \cite{no}, principal $SO(1, 2)$
subalgebras of hyperbolic Kac-Moody Lie algebras are constructed
and described in terms of the eigenvalues of their Casimir elements.
The principal $SO(1, 2)$ subalgebras are constructed
for certain indefinite type Lie algebras
which are not hyperbolic (\cite{gow}).

In this paper, we will construct
$SO(1, 2)$ subalgebras which are not principal
in rank-2 symmetric Kac-Moody Lie algebras.
However, we will actually construct $\lisl_2$ triples,
which are correspending to $SO(1, 2)$ subalgebras.
The discussion using $\lisl_2$ makes it easy to
compare with Dynkin-Kostant's theory of nilpotent orbits.
The principal $SO (1, 2)$ subalgebra is compatible with the
compact involution $\omega_0$.
This means that when we decompose $\lig$ as an
$\lisl_2$-module into a direct sum of irreducible components,
these irreducible components are unitary with respect to the
Hermitian form in \cite[\S 2.7]{kac} except the
$SO(1, 2)$ subalgebra itself.
We will search for $SO (1, 2)$ subalgebras
that are compatible with $\omega_0$.

In fact, we will search for $\lisl_2$-triples corresponding to
$SO (1, 2)$ subalgebras.
Since it is difficult to grasp the behavior of the imaginary root vectors,
we construct and classify the $\lisl_2$-triples
under the condition that
the nilpositive elements lie in the space spanned by the real roots.
This condition makes it possible to explicitly calculate them.
Although the meaning of this condition is unclear,
we remark that this condition is automatically satisfied for
principal $SO (1, 2)$ subalgebras.

We will construct most of the $\lisl_2$-triples
that satisfy these conditions.
In particular, we classified them in all cases
where the neutral element $H$ is dominant.
For these cases, we will calculate the weighted Dynkin diagrams
and the range of the eigenvalues on $\lih$ in the adjoint actions
of the Casimir elements.
We will also calculate some of the components that appear
when $\lig$ is decomposed by the action of each $\lisl_2$-triple.

\section{General theory of Kac-Moody Lie algebras}
In the following, we consider Kac-Moody Lie algebras on $\bbC$.
Denote by $\lig$ or $\lig (A)$ a Kac-Moody Lie algebra 
for the Cartan matrix $A$. Let $A$ be an $n \times n$ matrix.
Let $\lih$ be a Cartan subalgebra of $\lig$.
The root space with respect to the root $\alpha$ is written as $\lig_\alpha$.
We denote the Chevalley generators of $\lig$ by
$e_i, f_i, h_i \; (i = 0, \ldots, n - 1)$, and
the simple roots of $\lig$ as $\alpha_i \; (i = 0, \ldots, n - 1)$.
In this case, $\lra{h_i, \alpha_j} = a_{ij}$.
We write $\lin^+$ for the subalgebra generated by $e_i$'s and
$\lin^-$ for the subalgebra generated by $f_i$'s.
We also denote by $\omega$ the Chevalley involution on $\lig$.
We denote by $\mcl{W}$ the Weyl group of $\lig$.

The Cartan matrix $A$ is called symmetrizable
when there exist an invertible diagonal matrix
$D = \diag (\vep_0, \ldots, \vep_{n - 1})$
and a symmetric matrix $B$ such that $A = DB$.
A Kac-Moody Lie algebra whose Cartan matrix is symmetrizable is
called a symmetrizable Lie algebra.
From \cite[Theorem 2.2]{kac}, a symmetrizable Lie algebra
has a $\bbC$-valued nondegenerate symmetric bilinear form
$(\cdot \mid \cdot)$ called the standard form.

We fix a real form $\lih_\bbR$ of $\lih$
and define the antilinear automorphism $\omega_0$ in $\lig$ by
\begin{align*}
  \omega_0 (e_i) &= -f_i,\\
  \omega_0 (f_i) &= -e_i \hsp (i = 0, \ldots, n - 1),\\
  \omega_0 (h) &= -h \hsp (h \in \lih_\bbR).
\end{align*}
The automorphism $\omega_0$ is called
the compact involution of $\lig$.
From \cite[\S 2.7]{kac}, when $\lig$ is symmetrizable, we can define the nondegenerate
Hermitian form $(\cdot \mid \cdot)_0$ on $\lig$ by
\begin{align*}
  (x \mid y)_0 = - (\omega_0 (x) \mid y).
\end{align*}
\section{$SO(1, 2)$ subalgebras in hyperbolic Kac-Moody Lie algebras}
In this section, we will briefly recall the theory of
$SO(1, 2)$ subalgebras in the hyperbolic Kac-Moody Lie algebras
from \cite{no}.
An $SO (1, 2)$ subalgebra of $\lig$ is the 3-dimensional subalgebra
spanned by the three non-zero elements
$J^+ \in \lin^+ ,\; J^- \in \lin^- ,\; J_3 \in \lih$, satisfying
\begin{align*}
  [J_3, J^\pm] &= \pm J^\pm,\\
  [J^+, J^-] &= - J_3.
\end{align*}
A representation of an $SO(1, 2)$ subalgebra is called unitary
if the representation space $V$ has an Hermitian scalar product
$(\cdot, \cdot)$ and satisfies the following two conditions:
\begin{newEnum}
  \item \label{unitary1} The actions of $J^+$ and $J^-$ are adjoint,
  and the action of $J_3$ is self-adjoint.
  That is, for any $x, y \in V$,
  \begin{align*}
    ([J^+, x], y) &= (x, [J^-, y]),\\
    ([J_3, x], y) &= (x, [J_3, y]).
  \end{align*}
  \item The Hermitian scalar product $(\cdot, \cdot)$ is
  positive definite.
\end{newEnum}

A hyperbolic Kac-Moody Lie algebra is a Kac-Moody Lie algebra
such that the Cartan matrix $A$ is indefinite type and symmetrizable,
the Dynkin diagram is connected, and any proper connected subdiagram is
of type finite or affine.

Put
\begin{align*}
  p_i = - \sum_j A_{ij}^{-1}.
\end{align*}
Since $p_i \geq 0$ for any $i$, we put
\begin{align*}
  n_i = \sqrt{p_i}.
\end{align*}
We may construct a principal $SO(1, 2)$ subalgebra
of the hyperbolic Kac-Moody Lie algebra $\lig$
as follows.
\begin{align*}
  J_3 &= - \sum_j p_j h_j,\\
  J^+ &= \sum_j n_j e_j,\\
  J^- &= \sum_j n_j f_j.
\end{align*}
When $\lig$ is decomposed into the direct sum of irreducible modules
by the adjoint action of
the principal $SO (1, 2)$ subalgebra, these irreducible modules
except for $SO (1, 2)$ itself are infinite-dimensional
and unitary (\cite{no}).

In the case of the indefinite type Kac-Moody Lie algebras
which are not hyperbolic type,
the principal $SO (1, 2)$ subalgebras
can be constructed in the same way
if $p_i \geq 0$ for any $i$ (\cite{gow}).

On the other hand, an $\lisl_2$-triple in $\lig$ is a subalgebra
of $\lig$ with three non-zero elements
$X \in \lin^+ ,\; Y \in \lin^- ,\; H \in \lih$ such that
\begin{align*}
  [H, X] &= 2X,\\
  [H, Y] &= -2Y,\\
  [X, Y] &= H.
\end{align*}
The subalgebra spanned by $X, Y, H$ in $\lig$ is called
an $\lisl_2$ subalgebra.
An $\lisl_2$-triple can be constructed from the generators of
an $SO (1, 2)$ subalgebra by setting
\begin{align*}
  J^+ &= \frac{1}{\sqrt{2}} X,\\
  J^- &= - \frac{1}{\sqrt{2}} Y,\\
  J_3 &= \frac{1}{2} H.
\end{align*}

In the principal $SO(1, 2)$ subalgebra,
we have $J^- = - \omega_0 (J^+)$.
\begin{lem}[cf. {\cite[\S 2.7 and Theorem 2.2]{kac}}]
  \label{omegaAdjoint}
  Suppose that $\lig$ is symmetrizable.
  For any $u \in \lig$, $\ad u$ and $- {\ad \omega_0 (u)}$ are adjoint
  to each other with respect to $(\cdot \mid \cdot)_0$.
  That is, for any $x, y \in \lig$,
  \begin{align*}
    ([u, x] \mid y)_0 = - (x \mid [\omega_0 (u), y])_0.
  \end{align*}
\end{lem}
\begin{proof}
  By the definition of the Hermitian form $(\cdot \mid \cdot)_0$,
  let $(\cdot \mid \cdot)$ be the standard form and we have
  \begin{align*}
    ([u, x] \mid y)_0 &= - (\omega_0([u, x]) \mid y)\\
    &= - ([\omega_0 (u), \omega_0 (x)] \mid y)\\
    &= ([\omega_0 (x), \omega_0 (u)] \mid y),\\
    - (x \mid [\omega_0 (u), y])_0 &= (\omega_0 (x) \mid [\omega_0 (u), y]).
  \end{align*}
  From \cite[Thm 2.2 a)]{kac}, the standard form is invariant and we have
  \begin{align*}
    ([\omega_0 (x), \omega_0 (u)] \mid y) = (\omega_0 (x) \mid [\omega_0 (u), y]).
  \end{align*}
  Therefore
  \begin{align*}
    ([u, x] \mid y)_0 = - (x \mid [\omega_0 (u), y])_0.
  \end{align*}
\end{proof}
From \cref{omegaAdjoint}, even in the case of non-principal
$SO(1, 2)$ subalgebras, if $J^- = - \omega_0 (J^+)$ holds,
then the unitarity condition \cref{unitary1} is satisfied
with respect to the Hermitian form $(\cdot \mid \cdot)_0$
when $\lig$ is considered as a representation space with adjoint actions.
We will show that the converse is also true.
\begin{lem}
  \label{SO12unitary}
  Suppose that the adjoint action of the $SO (1, 2)$ subalgebra
  on $\lig$ satisfies the unitary condition \cref{unitary1}
  for the Hermitian form $(\cdot \mid \cdot)_0$.
  Then $J^- = -\omega_0 (J^+)$.
\end{lem}
\begin{proof}
  Since $J^+$ and $- \omega_0 (J^+)$ are adjoint to each other
  and so are $J^+$ and $J^-$,
  for any $x \in \lig, h \in \lih$, we have
  \begin{align*}
    ([J^+, x], h)_0 &= (x, [- \omega_0 (J^+), h])_0\\
    &= (x, [J^-, h])_0.
  \end{align*}
  Since $(\cdot \mid \cdot)_0$ is nondegenerate, we have
  \begin{align*}
    [- \omega_0 (J^+), h] = [J^-, h],
  \end{align*}
  which implies that
  \begin{align*}
    [h, J^- + \omega_0 (J^+)] = 0.
  \end{align*}
  Since this holds for any $h \in \lih$, we have
  $J^- + \omega_0 (J^+) \in \lih$.
  However, from $[J_3, J^\pm] = \pm J^\pm$,
  we have $J^+, J^- \in \lin^+ \oplus \lin^-$,
  and hence $J^- + \omega_0 (J^+) \in \lin^+ \oplus \lin^-$.
  From the above, we have $J^- + \omega_0 (J^+) = 0$.
  This proves the lemma.
\end{proof}
Motivated from this, we consider only
$SO (1, 2)$ subalgebras that satisfy the condition
$J^- = - \omega_0 (J^+)$, which is rephrased as
$Y = \omega_0 (X)$ in terms of $\lisl_2$-triples.
\section{The real roots of a rank 2 symmetric hyperbolic Kac-Moody Lie algebra}
Let $a$ be an integer such that $a \geq 3$. Let $\lig$
be a Kac-Moody Lie algebra whose Cartan matrix is
\begin{align*}
  \begin{pmatrix}
    2 & -a\\
    -a & 2
  \end{pmatrix}.
\end{align*}
It is of the hyperbolic type.
Any real root can be expressed as
$w (\alpha_0)$ or $w (\alpha_1)$
for some $w \in \mcl{W}$ (\cite{kac}).
Let $r_0$ and $r_1$ be the fundamental reflections for
$\alpha_0$ and $\alpha_1$.
Since $\mcl{W}$ is generated as a group by $r_0$ and $r_1$ (\cite{kac}),
any element of $\mcl{W}$ can be written in the form
\begin{align*}
  \begin{aligned}
    & (r_0 r_1)^i ,\; r_1 (r_0 r_1)^i,\\
    & (r_1 r_0)^i ,\; r_0 (r_1 r_0)^i.
  \end{aligned}
  \hsp (i \in \bbZ)
\end{align*}
\begin{lem}[{\cite[Proposition 4.4]{km}}]
  \label{posroot}
  Let $\{ F_n \}$ be a sequence determined by
  \begin{align*}
    F_0 = 0 ,\; F_1 = 1 ,\; F_{k + 2} = a F_{k + 1} - F_k.
  \end{align*}
  The real positive roots of $\lig$ are of the form either
  \begin{align*}
    \alpha = F_{k + 1} \alpha_0 + F_k \alpha_1
  \end{align*}
  or
  \begin{align*}
    \beta = F_k \alpha_0 + F_{k + 1} \alpha_1.
  \end{align*}
\end{lem}
We will distinguish between these roots and call them $\alpha$-type roots
and $\beta$-type roots, and define $\alpha$-type and $\beta$-type
for the root vectors as well.
\section{$\lisl_2$-triples of rank 2 hyperbolic Lie algebra which are compatible with compact involution}
Let $\lig$ be a rank 2 hyperbolic Kac-Moody algebra.
We want to find $X \in \lig$ in the space
spanned by the real root vectors
such that $X$, $Y = \omega_0 (X)$, and $H = [X, Y]$
form an $\lisl_2$-triple.
For $X \in \lig$ in the space spanned by the real root vectors,
$X$ can be written as
\begin{align*}
  X = \sum_i c_i E_i, \hsp (i \in \{ 0, \ldots, n_X - 1 \} ,\; c_i \in \bbC ,\; c_i \neq 0 ,\;
  E_i \in \lig_{\beta_i}, E_i \neq 0)
\end{align*}
where $\beta_i$ ($i \in \{ 0, \ldots, n_X - 1 \}$) are distinct real roots
and $n_X$ is a positive integer.
We call the $n_X$ the \textbf{length} of $X$.
First, we show the following lemma.
\begin{lem}
  \label{manyRealRootsDoNotMakesl2}
  If the length of $X$ is greater than or equal to 3,
  a required $\lisl_2$-triple does not exist.
\end{lem}
\begin{proof}
  We plot the roots on the $xy$-plane with the $\alpha_0$ component
  as $x$ coordinate
  and the $\alpha_1$ component as $y$ coordinate.
  If $X$, $Y = \omega_0 (X)$, and $H = [X, Y]$
  form an $\lisl_2$-triple,
  $[H, E_i] = 2E_i$ holds for each $i$.
  Put $\beta_i = p \alpha_0 + q \alpha_1$ where $p, q \in \bbR$. Since
  \begin{align*}
    [H, E_i] &= (p \alpha_0 + q \alpha_1)(H)E_i\\
    &= \left( p \alpha_0 (H) + q \alpha_1 (H) \right) E_i,
  \end{align*}
  we have
  \begin{align*}
    p \alpha_0 (H) + q \alpha_1 (H) = 2.
  \end{align*}
  Since we can write $H = r h_0 + s h_1$
  where $r, s \in \bbC$, $h_0, h_1$ are the part of Chevalley generator,
  we have
  \begin{align*}
    \begin{pmatrix}
      r & s
    \end{pmatrix}
    \begin{pmatrix}
      2 & -a\\
      -a & 2
    \end{pmatrix}
    \begin{pmatrix}
      p\\
      q
    \end{pmatrix}
    = 2.
  \end{align*}
  This represents a line on the $xy$-plane.
  In other words, $\beta_i$'s are colinear on the $xy$-plane.

  On the other hand, from \cite[Corollary 4.3]{km},
  the set of real roots $\Deltare$ is represented as the set of
  grid points on the hyperbola
  \begin{align*}
    \Deltare = \{ (x, y) \in \bbZ \times \bbZ \mid x^2 - axy + y^2 = 1 \}.
  \end{align*}
  Therefore, $\beta_i$'s are on the intersection of
  this hyperbola and the line.
  However, since there are at most two intersections of
  a hyperbola and a line, it is not possible to create
  the desired $\lisl_2$-triple when $n_X \geq 3$.
\end{proof}
From \cref{manyRealRootsDoNotMakesl2}, we only need to consider
the case when the length of $X$ is 1 or 2.
The multiplicity of real root is always 1(cf. \cite[Prop 5.1 a)]{kac}),
and the real root can be obtained by acting on the simple roots
with an element of the Weyl group.
Therefore, the real root vectors cam be written in the form
$c w (e_0)$ or $c w (e_1)$, using
$c \in \bbC ,\; w \in \mcl{W}$.
Note that the actions of the Weyl group on the element of $\lig$,
which as now written as $w (e_0)$s,
is the elements of $\Aut \lig$, which is determined by defining
\begin{align*}
  r_i (x) = (\exp (\ad f_i)(\exp (\ad - e_i))(\exp (\ad f_i))
\end{align*}
for the fundamental reflections $r_i$ ($i \in \{0, 1\}$)
(cf. \cite[Lemma 3.8]{kac}), and
$w (x) \in \lig_{w (\alpha)}$ holds for
$x \in \lig_\alpha$.
We now show the following lemma.
\begin{lem}
  \label{commute}
  For $w \in \mcl{W}$, $w \omega_0 = \omega_0 w$.
\end{lem}
\begin{proof}
  It is sufficient to show that when $w = r_i$.
  we will show $r_i^{-1} \omega_0 r_i \omega_0^{-1} = \id$.
  We have
  \begin{align*}
    r_i^{-1} &= (\exp (\ad - f_i)(\exp (\ad e_i))(\exp (\ad - f_i)),\\
    \omega_0 r_i \omega_0^{-1} &= (\exp (\ad \omega_0 (f_i))
    (\exp (\ad \omega_0 (- e_i)))(\exp (\ad \omega_0 (f_i)))\\
    &= (\exp (\ad - e_i))(\exp (\ad f_i))(\exp (\ad - e_i)).
  \end{align*}
  If we consider this in terms of the $SL(2, \bbC)$ representation
  to $\Aut \lig$, we get
  \begin{align*}
    \begin{pmatrix}
      1 & 1\\
      0 & 1
    \end{pmatrix}
    &\mapsto \exp \ad e_i,\\
    \begin{pmatrix}
      1 & -1\\
      0 & 1
    \end{pmatrix}
    &\mapsto \exp \ad - e_i,\\
    \begin{pmatrix}
      1 & 0\\
      1 & 1
    \end{pmatrix}
    &\mapsto \exp \ad f_i,\\
    \begin{pmatrix}
      1 & 0\\
      -1 & 1
    \end{pmatrix}
    &\mapsto \exp \ad - f_i.\\
  \end{align*}
  Therefore we have
  \begin{align*}
    \begin{pmatrix}
      1 & 0\\
      -1 & 1
    \end{pmatrix}
    \begin{pmatrix}
      1 & 1\\
      0 & 1
    \end{pmatrix}
    \begin{pmatrix}
      1 & 0\\
      -1 & 1
    \end{pmatrix}
    \begin{pmatrix}
      1 & -1\\
      0 & 1
    \end{pmatrix}
    \begin{pmatrix}
      1 & 0\\
      1 & 1
    \end{pmatrix}
    \begin{pmatrix}
      1 & -1\\
      0 & 1
    \end{pmatrix}
    \mapsto r_i^{-1} \omega_0 r_i \omega_0^{-1}.
  \end{align*}
  Calculating the left-hand side, we get
  \begin{align*}
    \begin{pmatrix}
      1 & 0\\
      -1 & 1
    \end{pmatrix}
    \begin{pmatrix}
      1 & 1\\
      0 & 1
    \end{pmatrix}
    \begin{pmatrix}
      1 & 0\\
      -1 & 1
    \end{pmatrix}
    \begin{pmatrix}
      1 & -1\\
      0 & 1
    \end{pmatrix}
    \begin{pmatrix}
      1 & 0\\
      1 & 1
    \end{pmatrix}
    \begin{pmatrix}
      1 & -1\\
      0 & 1
    \end{pmatrix}
    =
    \begin{pmatrix}
      1 & 0\\
      0 & 1
    \end{pmatrix}.
  \end{align*}
  Therefore, we have $r_i^{-1} \omega_0 r_i \omega_0^{-1} = \id$.
\end{proof}
\begin{lem}
  \label{singleRealRootDoNotMakessl2}
  When the length of $X$ is 1, a required $\lisl_2$ triple does not
  exist.
\end{lem}
\begin{proof}
  In the following, we denote the complex conjugate of
  a complex number $z$ by $\ovl{z}$ and the absolute value
  by $\norm{z}$.
  When the length of $X$ is 1,
  we can write $X = c w (e_0)$ or $X = c w (e_1)$
  for some $c \in \bbC$ and $w \in \mcl{W}$.
  When $X = c w (e_0)$ holds, from \cref{commute},
  \begin{align*}
    Y &= \ovl{c} w (\omega_0 (e_0))\\
    &= - \ovl{c} w (f_0),
  \end{align*}
  and
  \begin{align*}
    H &= - \norm{c}^2 w (h_0).
  \end{align*}
  To satisfy the condition that $X, Y, H$ forms $\lisl_2$-triple,
  $[H, X] = 2X$ should hold. Now we have
  \begin{align*}
    [H, X] &= [- \norm{c}^2 w(h_0), c w (e_0)]\\
    &= - \norm{c}^2 2cw(e_0).
  \end{align*}
  Therefore we have $\norm{c}^2 = -1$.
  Since there is no such complex number,
  the $\lisl_2$-triple cannot be constructed in this case.
  Since the same is true for $X = c w (e_1)$,
  the $\lisl_2$-triple cannot be constructed
  when $n_X = 1$.
\end{proof}
Next, we consider the case where the length of $X$ is 2.
First, consider the case when $X = c_0 w (e_0) + c_1 w' (e_1)$
for some $c_0, c_1 \in \bbC$ and $w, w' \in \mcl{W}$.
Let $k, l, m$, and $n$ be integers such that
$w (e_0) \in \lig_{k \alpha_0 + l \alpha_1} ,\; w' (e_1) \in \lig_{m \alpha_0 + n \alpha_1}$.
Using \cref{commute}, we have
\begin{align*}
  Y &= \ovl{c_0} w (- f_0) + \ovl{c_1} w' (- f_1)\\
  &= - \ovl{c_0} w (f_0) - \ovl{c_1} w' (f_1),
\end{align*}
and
\begin{align*}
  H &= [X, Y]\\
  &= - \norm{c_0}^2 w(h_0) - \norm{c_1}^2 w' (h_1)
  - c_0\ovl{c_1}[w(e_0), w'(f_1)] - \ovl{c_0}c_1[w'(e_1), w(f_0)]\\
  &= - \norm{c_0}^2 (k h_0 + l h_1) - \norm{c_1}^2 (m h_0 + n h_1)
  - c_0\ovl{c_1}[w(e_0), w'(f_1)] - \ovl{c_0}c_1[w'(e_1), w(f_0)].
\end{align*}
For the root space to which $w'(e_1)$ and $w(f_0)$ belong,
the sum of their roots is not 0.
The same is true for $w(e_0)$ and $w'(f_1)$.
Therefore, for $H \in \lih$ to hold,
it should hold that
\begin{align*}
  [w(e_0), w'(f_1)] = 0 ,\;
  [w'(e_1), w(f_0)] = 0.
\end{align*}
This condition holds when
$w = (r_0 r_1)^x ,\; w' = (r_1 r_0)^y$, or
$w = r_1 (r_0 r_1)^x ,\; w' = r_0 (r_1 r_0)^y$.
When this condition holds, we have
\begin{align*}
  H &= - \norm{c_0}^2 (k h_0 + l h_1) - \norm{c_1}^2 (m h_0 + n h_1)\\
  &= (- k \norm{c_0}^2 - m \norm{c_1}^2) h_0 + (- l \norm{c_0}^2 - n \norm{c_1}^2) h_1.
\end{align*}
If $[H, X] = 2X$, then $X, Y, H$ form an $\lisl_2$ triple.
From the fact that
\begin{align*}
  [H, X] &= (-k \norm{c_0}^2 - m \norm{c_1}^2) (2k - al) c_0 w (e_0) + (-l \norm{c_0}^2 - n \norm{c_1}^2) (-ak + 2l) c_0 w (e_0)\\
  &+ (-k \norm{c_0}^2 - m \norm{c_1}^2) (2m - an) c_1 w' (e_1) + (-l \norm{c_0}^2 - n \norm{c_1}^2) (-am + 2n) c_1 w' (e_1),
\end{align*}
it should be satisfied that
\begin{align*}
  (akl - 2k^2 + akl - 2l^2) \norm{c_0}^2 + (alm - 2km + akn - 2ln) \norm{c_1}^2 &= 2, \text{and}\\
  (akn - 2km + alm - 2ln) \norm{c_0}^2 + (amn - 2m^2 + amn - 2n^2) \norm{c_1}^2 &= 2.
\end{align*}
Let
\begin{align*}
  A &= 2akl - 2k^2 - 2l^2,\\
  B &= alm - 2km + akn - 2ln,\\
  C &= 2amn - 2m^2 - 2n^2,
\end{align*}
then
\begin{align*}
  A \norm{c_0}^2 + B \norm{c_1}^2 &= 2,\\
  B \norm{c_0}^2 + C \norm{c_1}^2 &= 2.
\end{align*}
If $B^2 - AC \neq 0$,
\begin{align*}
  \norm{c_0}^2 = \frac{2 (B - C)}{B^2 - AC} ,\; \norm{c_1}^2 = \frac{2 (B - A)}{B^2 - AC}.
\end{align*}
In addition, the following lemma holds.
\begin{lem}
  \label{B^2-AC}
  $B^2 - AC = (a^2 - 4)(kn - lm)^2$.
\end{lem}
\begin{proof}
  \begin{align*}
    B^2 - AC &= a^2l^2m^2 + 4k^2m^2 + a^2k^2n^2 + 4l^2n^2\\
    & \hphantom{{}={}} - 4aklm^2 + 2a^2klmn - 4al^2mn - 4ak^2mn + 8klmn -4akln^2\\
    & \hphantom{{}={}} - 4a^2klmn + 4aklm^2 + 4akln^2 + 4ak^2mn - 4k^2m^2\\
    & \hphantom{{}={}} - 4k^2n^2 + 4al^2mn - 4l^2m^2 - 4l^2n^2\\
    &= (-2a^2 + 8)klmn + (a^2 - 4)l^2m^2 + (a^2 - 4) k^2n^2\\
    &= (a^2 - 4)(kn - lm)^2
  \end{align*}
\end{proof}
From \cref{B^2-AC} and $a > 3$, if $kn - lm \neq 0$,
then $B^2 - AC > 0$.
Since we have $\norm{c_0}^2 \geq 0 ,\; \norm{c_1}^2 \geq 0$,
we want to find the condition.
If $w = (r_0 r_1)^x ,\; w' = (r_1 r_0)^y$, then
$k = F_{2x + 1} ,\; l = F_{2x} ,\; m = F_{2y} ,\; n = F_{2y + 1}$.
If $w = r_1 (r_0 r_1)^x ,\; w' = r_0 (r_1 r_0)^y$, then
$k = F_{2x + 1} ,\; l = F_{2x + 2} ,\; m = F_{2y + 2} ,\; n = F_{2y + 1}$.
We will show the following lemma first.
\begin{lem}
  \label{ACis-2}
  For any $i, j \in \bbZ_{\geq 0}$, if
  $k = F_{i + 1} ,\; l = F_i ,\; m = F_j ,\; n = F_{j + 1}$,
  then $A = C = -2$.
\end{lem}
\begin{proof}
  If we set $f(i) = 2aF_{i + 1}F_i - 2F_{i + 1}^2 - 2F_i^2$, then
  \begin{align*}
    f (i + 1) &= 2aF_{i + 2}F_{i + 1} - 2F_{i + 2}^2 - 2F_{i + 1}^2\\
    &= 2a(aF_{i + 1} - F_i)F_{i + 1} - 2(aF_{i + 1} - F_i)^2 - 2F_{i + 1}^2\\
    &= 2aF_{i + 1}F_i - 2F_{i + 1}^2 - 2F_i^2\\
    &= f(i).
  \end{align*}
  Therefore, $f(i) = f(0) = -2$.
  From this we know that $A = C = -2$.
\end{proof}
From \cref{ACis-2}, we have
\begin{align*}
  \norm{c_0}^2 = \norm{c_1}^2 &= \frac{2(B + 2)}{B^2 - 4}\\
  &= \frac{2}{B - 2}.
\end{align*}
\begin{lem}
  \label{BisGreaterThan2}
  For any $i, j \in \bbZ_{\geq 0}$, if
  $k = F_{i + 1} ,\; l = F_i ,\; m = F_j ,\; n = F_{j + 1}$,
  then $B > 2$.
\end{lem}
\begin{proof}
  Since $B$ depends on $i$ and $j$,
  we will write it subscripted as $B_{ij}$.
  Similarly, we wrill write $k, l, m, n$
  as $k_i, l_i, m_j, n_j$.
  We can calculate $B_{00} = a > 2$.
  It is sufficient to show that $B_{ij}$ is
  monotonically increasing with respect to $i$ and $j$.
  By symmetry, it is sufficient to show only for $i$.
  Since
  \begin{align*}
    k_{i + 1} = ak_i - l_i, \hsp l_{i + 1} = k_i,
  \end{align*}
  We can calculate
  \begin{align*}
    B_{(i + 1)j} - B_{ij} &= ak_im_j - 2(ak_i - l_i)m_j + a(ak_i - l_i)n_j - 2k_in_j\\
    &= (a - 2)k_im_j + 2l_im_j + (a^2 - 2)k_in_j - al_in_j\\
    &> (a - 2)k_im_j + 2l_im_j + (2a - 2)k_in_j - al_in_j \hsp (a \geq 3)\\
    &= (a - 2)k_im_j + 2l_im_j + (a - 2)k_in_j + a(k_i - l_i)n_j\\
    &> 0. \hsp (a \geq 3, k_i > l_i)
  \end{align*}
  This shows monotonicity.
\end{proof}
If we set $i = 2x$ and $j = 2y$ in \cref{BisGreaterThan2},
we have $B > 2$ when
$w = (r_0 r_1)^x ,\; w' = (r_1 r_0)^y$.
Similarly, if we swap $k$ and $m$, $l$ and $n$
in \cref{BisGreaterThan2} and set $i = 2x + 1$ and $j = 2y + 1$,
we can see that $B > 2$ is also true
when $w = r_1 (r_0 r_1)^x ,\; w' = r_0 (r_1 r_0)^y$.
\begin{prop}
  \label{existLisl2}
  Let $E_0, E_1$ be real positive roots of different types
  (in the sense of $\alpha$ and $\beta$ types).
  If we take the appropriate $c_0, c_1 \in \bbC$,
  then $X = c_0 E_0 + c_1 E_1 ,\; Y = \omega_0 (X) ,\; H = [X, Y]$
  form an $\lisl_2$-triple.
\end{prop}
\begin{proof}
  We can assume that $E_0$ is of type $\alpha$
  and $E_1$ is of type $\beta$.
  The $E_0$'s can be written in the form
  $c w (e_0)$ or $c w (e_1)$ using the constants $c$
  and $w \in \mcl{W}$.
  Considering $c_0$ and $c_1$, we only need to consider
  the root vector of the form $w (e_0)$ or $w (e_1)$.
  $(r_0 r_1)^i (e_0)$ and $r_0 (r_1 r_0)^i (e_1)$
  are of type $\alpha$, while
  $(r_1 r_0)^i (e_1)$ and $r_1 (r_0 r_1)^i (e_0)$
  are of type $\beta$.

  When $E_0 = (r_0 r_1)^i (e_0) ,\; E_1 = (r_1 r_0)^i (e_1)$ or
  when $E_0 = r_0 (r_1 r_0)^i (e_1) ,\; E_1 = r_1 (r_0 r_1)^i (e_0)$,
  we have $B > 2$. Therefore we can take
  $c_0 = c_1 = \sqrt{\frac{2}{B - 2}}$
  and we obtain the conclusion.
  When $E_0 = (r_0 r_1)^i (e_0) ,\; E_1 = r_1 (r_0 r_1)^i (e_0)$ or
  when $E_0 = r_0 (r_1 r_0)^i (e_1) ,\; E_1 = (r_1 r_0)^i (e_1)$,
  we can repeat the discussion of this section in the same way.
\end{proof}
The contents of this section can be summarized as follows.
\begin{thm}
  \label{mainThm}
  Let $X$ be an element of a space spanned by a
  real positive root vector.
  \begin{enumerate}
    \item If the length of $X$ is not 2,
    then $X ,\; Y = \omega_0(X) ,\; H = [X, Y]$ do not form
    an $\lisl_2$-triple.
    \item Suppose the length of $X$ is 2 and
    $E_0, E_1$ are real positive roots of different types
    (in the sense of $\alpha$-type and $\beta$-type).
    If we take appropriate $c_0, c_1 \in \bbC$, then
    $X = c_0 E_0 + c_1 E_1 ,\; Y = \omega_0(X) ,\; H = [X, Y]$
    forms an $\lisl_2$-triple.
  \end{enumerate}
\end{thm}
\begin{proof}
  We can see from \cref{manyRealRootsDoNotMakesl2},
  \cref{singleRealRootDoNotMakessl2}, and \cref{existLisl2}.
\end{proof}
\section{weighted Dynkin diagrams}
The Dynkin diagram of the Kac-Moody Lie algebra
we are dealing with is as follows.
\begin{center}
  \begin{tikzcd}[row sep = 0.5ex]
    \circ \ar[r, leftrightarrow, "a"'] & \circ
  \end{tikzcd}
\end{center}
In fact, the two vertices are connected by $a$ line segments,
which are abbreviated as shown in the figure above.
In this section, we compute the weighted Dynkin diagram
corresponding to the $\lisl_2$-triple constructed
in the previous section.
Weighted Dynkin diagram is a Dynkin diagram where each vertex is
labeled, and the label of vertex $i$ is defined as $\alpha_i (H)$.
Since the rank of the Kac-Moody Lie algebra we are considering
is 2, $i = 0, 1$.

In the following, let
$k = F_{i + 1} ,\; l = F_i ,\; m = F_j ,\; n = F_{j + 1}$.
Rewrite $X = c_0 E_0 + c_1 E_1$ and let
$E_0 \in \lig_{k \alpha_0 + l \alpha_1} ,\; E_1 \in \lig_{m \alpha_0 + n \alpha_1}$.
$E_0$ is a root vector of type $\alpha$ and
$E_1$ is a root vector of type $\beta$.
Now if $E \in \lig_{x \alpha_0 + y \alpha_1}$, then
\begin{align*}
  [H, E] &= (x \alpha_0 + y \alpha_1) (H) E\\
  &= (x \alpha_0 (H) + y \alpha_1 (H)) E.
\end{align*}
Recall that $[H, X] = 2X$.
Since $X = c_0 E_0 + c_1 E_1$, we have
\begin{align*}
  [H, E_0] &= 2 E_0,\\
  [H, E_1] &= 2 E_1.
\end{align*}
From $E_0 \in \lig_{k \alpha_0 + l \alpha_1} ,\; E_1 \in \lig_{m \alpha_0 + n \alpha_1}$,
it follows that
\begin{align*}
  \left\{
  \begin{aligned}
    k \alpha_0 (H) + l \alpha_1 (H) = 2\\
    m \alpha_0 (H) + n \alpha_1 (H) = 2.
  \end{aligned}
  \right.
\end{align*}
Solving this gives
\begin{align*}
  \alpha_0 (H) &= \frac{2 (n - l)}{kn - lm},\\
  \alpha_1 (H) &= \frac{2 (k - m)}{kn - lm}.
\end{align*}
Since $k > l$ and $n > m$, we have $kn - lm > 0$.
Therefore, the weighted Dynkin diagram is as follows.
\begin{center}
  \begin{tikzcd}[row sep = 0.5ex]
    \disp\frac{2(F_{j + 1} - F_i)}{F_{i + 1}F_{j + 1} - F_i F_j} &
    \disp\frac{2(F_{i + 1} - F_j)}{F_{i + 1}F_{j + 1} - F_i F_j}\\
    \circ \ar[r, leftrightarrow, "a"'] & \circ
  \end{tikzcd}
\end{center}
In the general Kac-Moody Lie algebra,
we say $h \in \lih$ is dominant if $h$ satisfies
$\alpha_i (h) \geq 0$ for any $i$.
In the finite type case, for any $\lisl_2$ triple,
we can transform it by the action of
appropriate element of the Weyl group
so that $H$ is dominant.
In the case of a rank 2 hyperbolic Kac-Moody Lie algebras,
if $H$ is dominant, then when we decompose $\lig$
into the direct sum of eigenspaces with respect to the adjoint action
of $H$, the dimension of the eigenspaces
corresponding to each eigenvalue will be finite.

With this motivation, we aim to classify
cases where $H$ is dominant.
When $H$ is dominant, each weight in the weighted Dynkin diagram
is greater than or equal to 0.
We do not consider whether an $\lisl_2$-triple can be constructed
when $X$ can be written as the sum of two root vectors
of the same type(in the sense of $\alpha$- and $\beta$- types),
but even if it can, we can show that $H$ is not dominant.
\begin{lem}
  \label{sameTypeVectorsDoNotMakeHDominantLem}
  Suppose $X = c_0 E_0 + c_1 E_1$ and that $E_0$ and $E_1$
  are real root vectors of the same type
  (in the sense of $\alpha$- and $\beta$- type)
  and the root to which $E_0$ and $E_1$ belong is different.
  If $X ,\; Y = \omega_0 (X) ,\; H = [X, Y]$ form
  an $\lisl_2$-triple, then $H$ is not dominant.
\end{lem}
\begin{proof}
  Suppose that $E_0$ and $E_1$ are both of $\alpha$-type.
  Let $k = F_{i + 1} ,\; l = F_i ,\; m = F_{j + 1} ,\; n = F_j$,
  and write
  $E_0 \in \lig_{k \alpha_0 + l \alpha_1} ,\; E_1 \in \lig_{m \alpha_0 + n \alpha_1}$.
  The roots to which $E_0$ and $E_1$ belong are different.
  Therefore $i \neq j$.
  Repeating the above calculation when $E_0$ is of $\alpha$-type and
  $E_1$ is of $\beta$-type, we get
  \begin{align*}
    \alpha_0 (H) &= \frac{2 (n - l)}{kn - lm},\\
    \alpha_1 (H) &= \frac{2 (k - m)}{kn - lm}.
  \end{align*}
  However, $n - l = F_j - F_i$ and $k - m = F_{i + 1} - F_{j + 1}$ are
  both not equal to zero and have different signs.
  Therefore, either $\alpha_0 (H)$ or $\alpha_1 (H)$
  will be negative, and $H$ will not be dominant.
\end{proof}
\begin{lem}
  \label{dominantLem}
  Of the $\lisl_2$-triples created by \cref{existLisl2},
  $H$ is dominant if and only if
  $i = j - 1 ,\; j ,\; j + 1$.
\end{lem}
\begin{proof}
  In order for $\alpha_0 (H) \geq 0$ to be true,
  $n - l = F_{j + 1} - F_i \geq 0$ should be true.
  Thus we have $i \leq j + 1$.

  Also, in order for $\alpha_1 (H) \geq 0$ to be true,
  $k - m = F_{i + 1} - F_j \geq 0$ should be true.
  Thus we have $i + 1 \geq j$.

  Putting these together, we get $j - 1 \leq i \leq j + 1$.
  The converse is obvious.
\end{proof}
\begin{prop}
  \label{dominantProp}
  The $\lisl_2$-triple constructed in \cref{existLisl2} can be
  transformed under the appropriate action of an element of Weyl gorup
  so that $H$ is dominant.
\end{prop}
\begin{proof}
  If $H$ is originally dominant, there is no need to transform it.
  Otherwise, It becomes $\norm{i - j} \geq 2$.
  
  Now we know
  \begin{align*}
    e_0 &\in \lig_{F_1 \alpha_0 + F_0 \alpha_1}, \hsp \text{$\alpha$-type}\\
    r_1(e_0) &\in \lig_{F_1 \alpha_0 + F_2 \alpha_1}, \hsp \text{$\beta$-type}\\
    r_0r_1(e_0) &\in \lig_{F_3 \alpha_0 + F_2 \alpha_1}, \hsp \text{$\alpha$-type}\\
    &\vdots
  \end{align*}
  and
  \begin{align*}
    e_1 &\in \lig_{F_0 \alpha_0 + F_1 \alpha_1}, \hsp \text{$\beta$-type}\\
    r_0(e_1) &\in \lig_{F_2 \alpha_0 + F_1 \alpha_1}, \hsp \text{$\alpha$-type}\\
    r_1r_0(e_1) &\in \lig_{F_2 \alpha_0 + F_3 \alpha_1}. \hsp \text{$\beta$-type}\\
    &\vdots
  \end{align*}
  Since $E_0$ is of $\alpha$-type and $E_1$ is of $\beta$-type,
  we can write
  \begin{align*}
    \begin{aligned}
      E_0 &= r_0 r_1 r_0 \cdots r_p (r_{1 - p})
      \in \lig_{F_{i + 1} \alpha_0 + F_i \alpha_1},\\
      E_1 &= r_1 r_0 r_1 \cdots r_q (r_{1 - q})
      \in \lig_{F_j \alpha_0 + F_{j + 1} \alpha_1}.
    \end{aligned}
    \hsp (p, q = 0 \; \text{or} \; 1)
  \end{align*}
  If $r_0$ acts on these elements,
  \begin{align*}
    r_0 (E_0) &= \hphantom{r_1 r_0} r_1 r_0 \cdots r_p (r_{1 - p}) \in \lig_{F_{i - 1} \alpha_0 + F_i \alpha_1},\\
    r_0 (E_1) &= r_0 r_1 r_0 r_1 \cdots r_q (r_{1 - q}) \in \lig_{F_{j + 2} \alpha_0 + F_{j + 1} \alpha_1}.
  \end{align*}
  Keeping in mind that $r_0 (E_0)$ is of $\beta$-type
  and $r_0 (E_1)$ is of $\alpha$-type,
  comparing these with the above, we can see that the number
  corresponding to $i$ is $j + 1$.
  Also, if $r_1$ acts on $E_0$ and $E_1$,
  \begin{align*}
    r_1 (E_0) &= r_1 r_0 r_1 r_0 \cdots r_p (r_{1 - p}) \in \lig_{F_{i + 1} \alpha_0 + F_{i + 2} \alpha_1},\\
    r_1 (E_1) &= \hphantom{r_0 r_1} r_0 r_1 \cdots r_q (r_{1 - q}) \in \lig_{F_j \alpha_0 + F_{j - 1} \alpha_1}.
  \end{align*}
  In this case, the number corresponding to $i$ is $j - 1$
  and the number corresponding to $j$ is $i + 1$.
  From this, $i - j$ becomes $j - i + 2$ under the action of $r_0$,
  and $j - i - 2$ under the action of $r_1$.
  If $i - j \geq 2$, $\norm{i - j}$ decreases by 2
  when $r_0$ is applied.
  If $j - i \geq 2$, $\norm{i - j}$ decreases by 2
  when $r_1$ is applied.
  By repeating this prosess, we can change $\norm{i - j}$ to
  0 or 1.
\end{proof}
Put together \cref{sameTypeVectorsDoNotMakeHDominantLem}
and \cref{dominantLem} to get the following.
\begin{thm}
  \label{dominantThm}
  The $\lisl_2$-triple $\{ X, Y, H \}$ of $\lig$,
  where $X$ is in the space spanned by real root vectors,
  $Y = \omega_0 (X)$, and $H$ is dominant,
  are all those listed in \cref{dominantLem}. \qed
\end{thm}
When $i = j$, the weighted Dynkin diagram is as follows.
\begin{center}
  \begin{tikzcd}[row sep = 0.5ex]
    \disp\frac{2}{F_{i + 1} + F_i} &
    \disp\frac{2}{F_{i + 1} + F_i}\\
    \circ \ar[r, leftrightarrow, "a"'] & \circ
  \end{tikzcd}
\end{center}
When $i = j + 1$, it is as follows.
\begin{center}
  \begin{tikzcd}[row sep = 0.5ex]
    0 & \disp\frac{2}{F_i}\\
    \circ \ar[r, leftrightarrow, "a"'] & \circ
  \end{tikzcd}
\end{center}
When $i = j - 1$, it is as follows.
\begin{center}
  \begin{tikzcd}[row sep = 0.5ex]
    \disp\frac{2}{F_{i + 1}} & 0\\
    \circ \ar[r, leftrightarrow, "a"'] & \circ
  \end{tikzcd}
\end{center}
\section{eigenvalues of the action of Casimir element on $\lih$}
The Casimir element $c$ of a finite-dimensional semisimple Lie
algebra $\lig_0$ is the element
\begin{align*}
  c = \sum_i x_i y_i \in U (\lig_0),
\end{align*}
where $(\cdot, \cdot)$ is the Killing form,
$\{ x_i \}$ is the basis of $\lig_0$, and
$\{ y_i \}$ is the dual basis with respect to this basis
and the Killing form.
$U (\lig_0)$ represents the universal enveloping algebra
of $\lig_0$.
When considering the action of $\lig_0$ on a $\lig_0$-module $L$,
the action of Casimir element is commutative with
any action of element of $\lig_0$.
When $L$ is an irreducible module,
from Schur's lemma, the action of Casimir element is
a scalar multiplication.

When decomposing $\lig$ by the action of
the $\lisl_2$ subalgebra created in the previous section,
from $\dim \lih = 2$, there are two irreducible modules
that pass through $\lih$.
In Particular, one is the $\lisl_2$ subalgebra itself.
In this section, we will find the eigenvalues of
the Casimir element when it acts on $\lih$ with
adjoint action, and we will determine how many times
the action of the Casimir element makes
the two irreducible module.
Let $c_L$ be the Casimir element in the $\lisl_2$ subalgebra
constructed above,
\begin{align*}
  c_L &= \frac{1}{8}H^2 + \frac{1}{4}XY + \frac{1}{4}YX\\
  &= \frac{1}{8}H^2 + \frac{1}{4}H + \frac{1}{2}YX.
\end{align*}
For the $\lisl_2$-triple constructed in \cref{existLisl2},
we will calculate the eigenvalue of the Casimir element.
Let $X = c_0 E_0 + c_1 E_1$, where $E_0$ is of type $\alpha$
and $E_1$ is of type $\beta$.
Using $p, q \in \{ 0, 1 \}$, we can write
$E_0 = w (e_{p}) ,\; E_1 = w' (e_{q})$.
Let $E_0 \in \lig_{k\alpha_0 + l\alpha_1} ,\; E_1 \in \lig_{m \alpha_0 + n \alpha_1}$,
we can write
$k = F_{i + 1} ,\; l = F_i ,\; m = F_j ,\; n = F_{j + 1}$.
We have
\begin{align*}
  c_L(h_0) &= \frac{1}{2} [Y, [X, h_0]]\\
  &= \frac{1}{2} [Y, [c_0 w(e_{p}) + c_1 w'(e_{q}), h_0]]\\
  &= \frac{1}{2} [Y, - (2k - al) c_0 w(e_{p}) - (2m - an) c_1 w'(e_{q})]\\
  &= \frac{1}{2} [(al - 2k) c_0 w(e_{p}) + (an - 2m) c_1 w'(e_{q}), \ovl{c_0} w(f_{p}) + \ovl{c_1} w'(f_{q})].
\end{align*}
Now since 
both $[w(e_{p}), w'(f_{q})]$ and $[w'(e_{q}), w(f_{p})]$ are 0,
we have
\begin{align*}
  c_L(h_0) &= \frac{1}{2} ((al - 2k) \norm{c_0}^2 w(h_{p}) + (an - 2m) \norm{c_1}^2 w'(h_{q}))\\
  &= \frac{1}{2} ((al - 2k) \norm{c_0}^2 (kh_0 + lh_1) + (an - 2m) \norm{c_0}^2 (mh_0 + nh_1))\\
  &= \frac{1}{2} ((al - 2k) k + (an - 2m) m) \norm{c_0}^2h_0
  + \frac{1}{2} ((al - 2k) l + (an - 2m) n) \norm{c_0}^2h_1.
\end{align*}
Similarly,
\begin{align*}
  c_L(h_1) &= \frac{1}{2} [Y, [X, h_1]]\\
  &= \frac{1}{2} [Y, [c_0 w(e_{p}) + c_1 w'(e_{q}), h_1]]\\
  &= \frac{1}{2} [Y, - (-ak + 2l) c_0 w(e_{p}) - (-am + 2n) c_1 w'(e_{q})]\\
  &= \frac{1}{2} [(ak - 2l) c_0 w(e_{p}) + (am - 2n) c_1 w'(e_{q}), \ovl{c_0} w(f_{p}) + \ovl{c_1} w'(f_{q})]\\
  &= \frac{1}{2} (ak - 2l) \norm{c_0}^2 w'(h_{p}) + (am - 2n) \norm{c_1}^2 w'(h_{q})\\
  &= \frac{1}{2} (ak - 2l) \norm{c_0}^2 (kh_0 + lh_1) + (am - 2n) \norm{c_0}^2 (mh_0 + nh_1)\\
  &= \frac{1}{2} ((ak - 2l) k + (am - 2n) m)\norm{c_0}^2h_0
  + ((ak - 2l) l + (am - 2n) n)\norm{c_0}^2h_1.
\end{align*}
The eigenvalues of the action of $c_L$ are
eigenvalues of
\begin{align*}
  \frac{1}{2} \norm{c_0}^2
  \begin{pmatrix}
    (al - 2k)k + (an - 2m)m & (ak - 2l)k + (am - 2n)m\\
    (al - 2k)l + (an - 2m)n & (ak ^ 2l)l + (am - 2n)n
  \end{pmatrix}.
\end{align*}
For simplicity, let
\begin{align*}
  P &= (al - 2k)k + (an - 2m)m,\\
  Q &= (ak - 2l)k + (am - 2n)m,\\
  R &= (al - 2k)l + (an - 2m)n,\\
  S &= (ak - 2l)l + (am - 2n)n.
\end{align*}
Solving for
\begin{align*}
  \begin{vmatrix}
    X - P & -Q\\
    -R & X - S
  \end{vmatrix}
  = 0,
\end{align*}
we get
\begin{align*}
  X &= \frac{(P + S) \pm \sqrt{(P + S)^2 - 4PS + 4QR}}{2}\\
\end{align*}
and the eigenvalues are
\begin{align*}
  \frac{(P + S) \pm \sqrt{(P + S)^2 - 4PS + 4QR}}{4} \norm{c_0}^2.
\end{align*}
we will write this as $E^\pm$.
\begin{lem}
  \label{qrps}
  $P + S = -4$ and $QR - PS = B^2 - 4$.
\end{lem}
\begin{proof}
  this lemma is shown from
  \begin{align*}
    P + S &= akl - 2k^2 + amn - 2m^2 + akl - 2l^2 + amn - 2n^2\\
    &= A + C\\
    &= -4\\
  \end{align*}
  and
  \begin{align*}
    QR &= (ak - 2l)(al - 2k)kl + (ak - 2l)(an - 2m)kn\\
    & \hphantom{{}={}} + (am - 2n)(al - 2k)lm + (am - 2n)(an - 2m)mn\\
    PS &= (al - 2k)(ak - 2l)kl + (al - 2k)(am - 2n)kn\\
    & \hphantom{{}={}} + (an - 2m)(ak - 2l)lm + (an - 2m)(am - 2n)mn\\
    QR - PS &= (ak - 2l)(an - 2m)(kn - lm) + (al - 2k)(am - 2n)(lm - kn)\\
    &= (a^2kn - 2akm - 2aln + 4lm - a^2lm + 2aln + 2akm - 4kn)(kn - lm)\\
    &= (a^2 - 4)(kn - lm)^2\\
    &= B^2 - AC \hsp (\text{from \cref{B^2-AC}})\\
    &= B^2 - 4.
  \end{align*}
\end{proof}
From \cref{qrps}, we have
\begin{align*}
  E^\pm &= \frac{(P + S) \pm \sqrt{(P + S)^2 + 4(QR - PS)}}{4}\norm{c_0}^2\\
  &= \frac{-4 \pm \sqrt{16 + 4(B^2 - 4)}}{4}\norm{c_0}^2\\
  &= \frac{-4 \pm 2 \norm{B}}{4}\norm{c_0}^2\\
  &= \frac{-2 \pm B}{2} \cdot \frac{2}{B - 2} \hsp (\text{from \cref{BisGreaterThan2}, $B > 2$})\\
  &= - \frac{B + 2}{B - 2} ,\; 1.
\end{align*}
The eigenvalue 1 corresponds to the $\lisl_2$ subalgebra itself,
and the other eigenvalue corresponds to the other
irreducible component. Let
\begin{align*}
  E^+ = - \frac{B + 2}{B - 2}
\end{align*}
and we will consider the range of values of $E^+$.
Since $B > 2$ from \cref{BisGreaterThan2},
$E^+$ is strictly increasing with respect to $B$.
Also, from \cref{BisGreaterThan2},
$B$ is strictly increasing
with respect to $i$ and $j$ respectively.
When $i = j = 0$, we have $B = a$ and
\begin{align*}
  E^+ = - \frac{a + 2}{a - 2}.
\end{align*}
When $B > 2$, we have $E^+ < -1$ and
\begin{align*}
  - \frac{a + 2}{a - 2} \leq E^+ < -1,
\end{align*}
We will do some more caluculations on the value of $B$.
\begin{lem}
  \label{Btoinfty}
  For any $i$,
  \begin{align*}
    \lim_{j \to \infty} B = \infty
  \end{align*}
  and for any $j$,
  \begin{align*}
    \lim_{i \to \infty} B = \infty.
  \end{align*}
\end{lem}
\begin{proof}
  From symmetry, only the first half needs to be shown.
  We can calculate that
  \begin{align*}
    B &= alm - 2km + akn - 2ln\\
    &= aF_im - 2F_{i + 1}m + aF_{i + 1}n - 2F_in\\
    &= F_{i + 1}(an - 2m) + aF_im - 2F_in\\
    &\geq F_{i + 1}(3n - 2m) + aF_im - 2F_in\\
    &= F_{i + 1}(n - 2m) + aF_im + 2n(F_{i + 1} - F_i).
  \end{align*}
  If $i \geq 1$, then $F_{i + 1} = aF_i - F_{i - 1}$
  and $F_i > F_{i - 1}$,
  and if $i = 0$, then $F_{i + 1} = aF_i$.
  From this, we have
  \begin{align*}
    F_{i + 1} > (a - 1)F_i
  \end{align*}
  and therefore
  \begin{align*}
    2n(F_{i + 1} - F_i) > 2(a - 2)nF_i.
  \end{align*}
  From $n > 2m$ and $a \geq 3$, we have
  \begin{align*}
    \lim_{j \to \infty} F_{i + 1}(n - 2m) &= \infty,\\
    \lim_{j \to \infty} aF_im &= \infty,\\
    \lim_{j \to \infty} 2(a - 2)nF_i &= \infty.
  \end{align*}
  Consequently, we have
  \begin{align*}
    \lim_{j \to \infty} B = \infty.
  \end{align*}
\end{proof}
From \cref{Btoinfty}, we have
\begin{align*}
  \lim_{i \to \infty} E^+ = -1, \hsp \lim_{j \to \infty} E^+ = -1.
\end{align*}
The above can be summarized as follows.
\begin{prop}
  \label{rangeOfE}
  $E^+$ is strictly increasing for $i, j$, and
  \begin{align*}
    - \frac{a + 2}{a - 2} \leq E^+ < -1
  \end{align*}
  and
  \begin{align*}
    \lim_{i \to \infty} E^+ = -1, \hsp \lim_{j \to \infty} E^+ = -1.
  \end{align*}
  \qed
\end{prop}
\begin{exam}
  Let $a = 3$. Then
  $F_0 = 0 ,\; F_1 = 1 ,\; F_2 = 3 ,\; F_3 = 8, \cdots$.
  The $E^+$'s for the $\lisl_2$ triples created in \cref{mainThm}
  are as in \Cref{table}.
  \begin{table}[H]
    \centering
    \caption{The $E^+$'s for the $\lisl_2$ triples created in \cref{mainThm}, $a = 3$}
    \label{table}
    \begin{tabular}{c|cccccc}
      $(k, l) \backslash (m, n)$ & (0, 1) & (1, 3) & (3, 8) & (8, 21) & (21, 55) & $\cdots$\\ \hline
      (1, 0)   & $-5$ & $-1.8$ & $-1.25$ & $-1.088889$ & $-1.033058$ &\\
      (3, 1)   & $-1.8$ & $-1.25$ & $-1.088889$ & $-1.033058$ & $-1.012500$ &\\
      (8, 3)   & $-1.25$ & $-1.088889$ & $-1.033058$ & $-1.012500$ & $-1.004756$ & $\cdots$\\
      (21, 8)  & $-1.088889$ & $-1.033058$ & $-1.012500$ & $-1.004756$ & $-1.001814$ &\\
      (55, 21) & $-1.033058$ & $-1.012500$ & $-1.004756$ & $-1.001814$ & $-1.000693$ &\\
      $\vdots$ & & & $\vdots$ & & & $\ddots$
    \end{tabular}
  \end{table}
\end{exam}
\section*{Acknowledgements}
I would like to express my appreciation to my supervisor,
Prof. Hisayosi Matumoto for his thoughtful guidance.
I am also grateful to my colleague Yuki Goto for meaningful advices on English.

\end{document}